\documentclass[a4paper]{article}

\usepackage[utf8]{inputenc} 	
\usepackage{amsthm,amsmath,amssymb, mathrsfs}

\usepackage{enumitem}
\usepackage[colorlinks = true]{hyperref}  
  
\usepackage[authoryear]{natbib}

\theoremstyle{plain}
\newtheorem{thm}{Theorem}
\newtheorem{lem}{Lemma}
\newtheorem{prop}{Proposition}
\newtheorem{cor}{Corollary}
\theoremstyle{remark}
\newtheorem{df}{Definition}
\newtheorem{ass}{Assumption}

\newtheorem{rem}{Remark}
\newtheorem{ex}{Example}

\newcommand{\norm}[1]{\left\|{#1}\right\|}
\newcommand{\normfty}[1]{\left\|{#1}\right\|_{\infty,\mu}}

\newcommand{\argmin}{\mathop{\rm argmin}}
\newcommand{\argmax}{\mathop{\rm argmax}}

\newcommand{\pen}{\mathop{\rm pen}\nolimits}
\newcommand{\Var}{\mathop{\rm Var}\nolimits}

\newcommand{\E}{{\mathbb{E}}}

\newcommand{\N}{{\mathbb{N}}}

\newcommand{\Q}{{\mathbb{Q}}} 
\newcommand{\R}{{\mathbb{R}}}

\newcommand{\Z}{{\mathbb{Z}}}

\newcommand{\sL}{{\mathscr{L}}} 
\newcommand{\sM}{{\mathscr{M}}}

\newcommand{\sP}{{\mathscr{P}}}

\DeclareMathAlphabet{\mathscrbf}{OMS}{mdugm}{b}{n}

\newcommand{\sbP}{{\mathscrbf{P}}}

\newcommand{\cA}{{\mathcal{A}}}
\newcommand{\cB}{{\mathcal{B}}}
\newcommand{\cC}{{\mathcal{C}}}

\newcommand{\cE}{{\mathcal{E}}}
\newcommand{\cF}{{\mathcal{F}}}

\newcommand{\cI}{{\mathcal{I}}}

\newcommand{\cM}{{\mathcal{M}}}

\newcommand{\cP}{{\mathcal{P}}}

\newcommand{\cS}{{\mathcal{S}}}

\newcommand{\cX}{{\mathcal{X}}}

\newcommand{\gh}{{\mathbf{h}}}

\newcommand{\gj}{{\mathbf{j}}}
\newcommand{\gk}{{\mathbf{k}}}

\newcommand{\gp}{{\mathbf{p}}}
 
\newcommand{\gr}{{\mathbf{r}}}

\newcommand{\gL}{{\mathbf{L}}} 
\newcommand{\gM}{{\mathbf{M}}}

\newcommand{\gP}{{\mathbf{P}}}

\newcommand{\gX}{{\mathbf{X}}}

\newcommand{\bs}[1]{\boldsymbol{#1}}

\newcommand{\bsX}{{\bs{X}}}

\newcommand{\gbeta}{\bs{\beta}}

\newcommand{\eps}{{\varepsilon}}
\newcommand{\et}{^{\star}}
\newcommand{\pol}[2]{\cP_{#1,#2}}
\newcommand{\poldir}[2]{\cP^{dir}_{#1,#2}}
\renewcommand{\deg}{\mathrm{deg}}

\newcommand{\1}{1\hskip-2.6pt{\rm l}}

\renewcommand{\ge}{\geqslant}

\begin{document}

\title{A model-based approach to density estimation in sup-norm}
\author{Guillaume Maillard}

\maketitle


\begin{abstract}
We define a general method for finding a quasi-best approximant in sup-norm to a target density belonging to a given model, based on independent samples drawn from distributions which average to the target (which does not necessarily belong to the model). We also provide a general method for selecting among a countable family of such models. These estimators satisfy oracle inequalities in the general setting. The quality of the bounds depends on the volume of sets on which $|p-q|$ is close to its maximum, where $p,q$ belong to the model (or possibly to two different models, in the case of model selection). This leads to optimal results in a number of settings, including piecewise polynomials on a given partition and anisotropic smoothness classes. Particularly interesting is the case of the single index model with fixed smoothness $\beta$, where we recover the one-dimensional rate: this was an open problem.
\end{abstract}





\section{Introduction}

In regression, classification and density estimation, the model-based approach to estimation \citep{Massart2014} consists in specifying a collection of \emph{models}, together with a standard method for performing estimation within each model and a \emph{penalty} or \emph{model selection criterion} for selecting among the models. In density estimation, this approach can for example be based on maximum likelihood or least-squares for estimating within a model \cite[Example 1]{Massart2014} and cross-validation for selecting among models. 

This leads to a number of desirable practical and theoretical properties. 
First, the approach is very flexible and general since usually, a wide variety of different model collections are compatible with the basic method.
Moreover, the analysis of the risk of model-based estimators naturally subdivides into an "approximation-theoretic" part dealing with the approximation properties of the model $m$, and a "statistical part" dealing with the difficulty of estimating within $m$, which can be solved separately \citep{Luxburg2011}. 
Appropriate penalties can be derived from concentration inequalities for (weighted) empirical processes \cite[Chapter 1]{Massart2007}. The resulting model selection estimators optimize the tradeoff between the approximation error and the penalty \cite[Section 3]{Barron1999}.
This naturally leads to minimax-adaptive estimators, provided the model collection is well chosen \cite[Section 1.4]{Barron1999}.  

In density estimation, the model-based approach has mainly been used together with least-squares and maximum likelihood methods which target the minimizer of the squared $L^2$ and Küllback-Leibler distance to the underlying density. However, if we are interested instead in some other distance, the least-squares or maximum likelihood estimates may be arbitrarily far from optimal, as remarked by Devroye and Lugosi in the case of the $L^1$ loss \cite[Chapter 6]{Devroye2001}.  One is then left with the task of devising a data-driven method to minimize the given distance $d$ over a model $m$. The difficulty here is that empirical risk  minimization cannot be used in general, for lack of a suitable contrast function.  This problem was first solved by \citet[Chapter 6]{Devroye2001} in the case of the $L^1$ loss and by \cite{Baraud2016rho_estim} in the case of the Hellinger distance. More recently, \cite{Baraud2021} devised a general strategy called $\ell-$estimation, which applies to all $L^p$ losses for $p \in [1,+\infty)$ (among others), but \emph{not} however to the $L^\infty$ distance in general. He also did not address the problem of model selection.

In this article, we treat the case of the sup-norm loss, establishing a general method for model-based density estimation in $L^\infty$. Our method results from the application of a variant of Baraud's $\ell-$estimation to a certain parametrized family of semi-norms approximating the essential supremum. The estimation error of this method depends mainly on the measure of sets on which elements of the model $m$ remain close to their extremal value. In addition, we develop an entirely data-driven method for model selection, using penalties derived from concentration inequalities.  

As an application, we consider the class of piecewise polynomial functions on dyadic partitions of $\mathbb{R}^d$ and show that the resulting model-based estimator is minimax-adaptive over classes of functions with anisotropic smoothness. Our result improves on what was previously known in the literature for non model-based estimators: not only does our estimator converge at the optimal rate (a property already established by \cite{Lepski2013} for his adaptive kernel method), it also depends optimally on the underlying density, up to a constant depending only on the dimension and the degree of the polynomials.

This article is structured as follows. First, the setting is introduced and the main notation is defined in Section \ref{sect-2}. The model-based estimator is defined in Section \ref{sec-est_1mod} and a general oracle inequality is established. This general result is applied to models of piecewise polynomials in Section \ref{sec-piec_poly-1mod}. A minimax lower bound establishes the optimality of our estimator up to logarithmic factors. Section \ref{sec:mod_select_method} addresses the model selection problem in the general setting, resulting in an estimator which satisfies an oracle inequality. In Section \ref{sec:dyad_mod_coll}, we consider the case of piecewise polynomials on \emph{regular dyadic partitions}, where one must select among such partitions, and show that our assumptions hold in that case. In Section \ref{sec-rates}, the resulting estimator is shown to be minimax-adaptive on classes of functions with anisotropic smoothness. A matching minimax lower bound is established, based on a result of \citet{Lepski2013}.  
Proofs of these results can be found in the Supplementary Material \citep{Proofs}.
%
\section{Setting and notation}\label{sect-2}
Let $(E,\cE,\mu)$ be a measure space, with $\sigma-$finite measure $\mu$. For any measurable function $f$, let
\begin{equation}
   \normfty{f} = \sup \left\{ r \geq 0 : \mu \left( \left\{ x \in E : |f(x)| \geq r \right\} \right) > 0 \right\}  
\end{equation}
denote its essential supremum and let 
$\sL_{\infty}(E,\mu)$ be the set of measurable functions $f$ on $(E,\cE,\mu)$ such that $\normfty{f} < +\infty$. Let $L_\infty(E,\mu)$ denote the associated set of equivalent classes for the relation of equality $\mu-$almost everywhere. The topic of this article is density estimation on $L^\infty(E,\mu)$ with respect to the norm $\normfty{\cdot}$.

We assume that the observations $X_{1},\ldots,X_{n}$ are independent but not necessarily i.i.d, which allows to consider possible outliers. Let $P_{1}\et,\ldots,P_{n}\et$ denote their marginals. We assume that the marginals have densities $p_{1}\et,\ldots,p_{n}\et$ belonging to $L_\infty(E,\mu)$ - otherwise, estimation in $L_\infty$ norm is impossible. Throughout this article,  $\gP\et=\bigotimes_{i=1}^{n}P_{i}\et$ denotes the distribution of the observation $\bsX=(X_{1},\ldots, X_{n})$, and $\gp \et = (p_{1}\et,\ldots,p_{n}\et)$ denotes the corresponding $n-$uplet of probability densities. Moreover, $P\et$ denotes the mixture distribution, $P\et = \frac{1}{n} \sum_{i = 1}^n P_i\et$, and $p\et$ denotes the corresponding probability density. In case the data is not truly i.i.d, the estimators considered in this article estimate $p\et$: in particular, they are robust to small departures from the i.i.d assumption (in the $L^\infty$ sense).

\subsection{Notation \label{sect-2.1}}
Bold capitals $\gP$ will be used to denote either the product measure $\gP=\bigotimes_{i=1}^{n}P_{i}$ or the $n$-uplet $(P_{1},\ldots,P_{n})$,
depending on the context. The notation $\E[g(\bsX)]$ is to be interpreted under the assumption that $\gX \sim \gP\et$, while $\E_{S}[f(X)]$ denotes the expectation of $f(X)$ when $X \sim S$. The same conventions apply to $\Var\left( g(\bsX)\right)$ and $\Var_{S}\left( f(X)\right)$. The same letter will always be used to denote a measurable function $q$ and the corresponding (signed) measure $Q = q d\mu$: lowercase letters refer  to functions and uppercase letters, to measures.

In addition, we shall use the following standard notation. For $x\in\R$, $x_{-} =\max\{0,-x\}$; for $x\in\R^{d}$, $B(x,r)$ denotes the closed Euclidean ball centered at $x$ with radius $r\ge 0$. For a positive integer $d$, $L_{\infty}(\R^{d})$ means $L_{\infty}(E,\mu)$ when $E=\R^{d}$, $\cE$ is the Borel $\sigma$-algebra and $\mu=\lambda$ is the Lebesgue measure on $\R^{d}$.

\subsection{Models and losses\label{sect-2.2}}
Denote by $\cM$ a collection of models $m$, each of which is a subset of $\sP = L_\infty(E,\mu) \cap L_1(E,\mu)$. For reasons of technical convenience, we do not impose that the models $m$ consist of densities. In the following, we will always assume that $\cM$, the model collection, as well as the models $m \in \cM$, are {\em at most countable} in order to avoid measurability issues. Let $\sM$ denote the union of all the models: $\sM = \cup_{m \in \cM} m$. In particular,$\sM$ is countable. Since most of the models used by statisticians are separable, this assumption is not restrictive in practice: one can always replace an uncountable, separable model $\overline{m}$ by a dense countable subset $m$, without changing the approximation error. 


Given the observation $\bsX$ and a model $m$, we want to design an estimator $\widehat p_m =\widehat p_m(\gX)$ of $p\et$ with values in $m$ which is as close as possible to $p\et$ in norm $\normfty{\cdot}$. 
Since $\widehat p_m \in m$ by definition, $\normfty{p\et - \widehat p_m}$ is lower bounded by 
\begin{equation}
   \inf_{q \in m}\normfty{q - p\et}=d_{\infty,\mu}(p\et,m), 
\end{equation}
the \emph{approximation error} of the model $m$ in $L^\infty(E,\mu)$.
The best that can be expected of $\widehat p_m$ is that $\normfty{p\et - \widehat p_m}$ be close to $d_{\infty,\mu}(p\et,m)$. This term cancels when $p\et \in \overline{m}$, where 
\begin{equation}
\overline{m}=\{p \in \sP\,| \inf_{q \in m} \normfty{p - q} = 0\},
\label{eq-overM}
\end{equation}
which generalizes the case of $\gp \et = (p,\ldots,p)$ for some $p \in m$ (the "true model" case).

\section{Estimator on a single model} \label{sec-est_1mod}
\subsection{Definition and properties in a general context}
To achieve model-based estimation in the norm $\normfty{\cdot}$, we adapt the general method of $\ell-$estimation introduced by \cite{Baraud2021}. For a given norm $\norm{\cdot}$ and $p,q \in m \subset B$ (where $B$ is a function space), this method relies on finding suitable measurable functions $g_{p,q}$ such that
\begin{itemize}
    \item $\int (p-q) g_{p,q} d\mu = \norm{p - q}$ 
    \item For all $f \in B, \int f g_{p,q} d\mu \leq \norm{f}$
    \item $\frac{1}{n} \sum_{i = 1}^n g_{p,q}(X_i)$ is close to its expectation (for all $p,q \in m$).
\end{itemize}
In the case of the norm $\normfty{\cdot}$ and the space $B = L^\infty(E,\mu)$, the first two requirements cannot be simultaneously satisfied in general, so we shall instead seek a suitable approximation of  $\normfty{p - q}$ by $\int (p-q) g_{p,q} d\mu$, for some  $g_{p,q}$ such that $\int |g_{p,q}| d\mu \leq 1.$ To that end, fix a VC class of measurable sets $\cC$, with VC-dimension $V$. For any $f \in L_1(E,\mu)$ and any $h > 0$, let
\begin{equation} \label{eq_def_norm_h}
   |f|_{h} = \sup_{C \in \cC} \frac{1}{\mu(C) + h} \left| \int_C f d\mu \right|. 
\end{equation}
This semi-norm is a norm whenever the sets of $\cC$ have finite measure and generate the Borel sigma-algebra: this will be the case with all the examples which we will consider.

\begin{rem}
    The parameter $h$ controls the quality of the approximation to $\normfty{\cdot}$. When $h = 0$,
    \[ |\cdot|_0 = \normfty{\cdot} \]
    for a sufficiently "nice" measure space $(E,\mu)$ and sufficiently rich class $\cC$. In particular this is the case when $(E,\mu)$ is a finite-dimensional space with Lebesgue measure and $\cC$ satisfies the assumptions of the Lebesgue differentiation theorem, see Proposition \ref{prop_consistence}. However, the supremum in the definition of $|f|_0$ typically is not realized on any set $C$ with $\mu(C) > 0$, unless $\argmax \{|f| \}$ has non-zero measure. Rather, the idea is to exploit the continuity at $0$:
    \[ \lim_{h \to 0} |f|_h  = |f|_0 = \normfty{f} \]
    for all $f \in L^\infty \cap L^1$, under appropriate conditions (see Proposition \ref{prop_consistence}). Thus, the role played by $h$ is similar to that of a bandwidth in kernel density estimation.
\end{rem}

For practical purposes, it is necessary to select, for each pair of candidate densities $p,q$ belonging to the model $m$, an appropriate "test set" $C_h(p,q)$ which approaches the supremum in equation \eqref{eq_def_norm_h}. The following definition quantifies the approximation error. 

\begin{df}
Let $h > 0$ and $C_h : \sP^2 \to \cC$ be a function. Let then
\begin{equation} \label{eq_lb_intCh}
   1 - \varepsilon_m(h,C_h) = \inf_{(p,q) \in m^2} \left\{ \frac{\left|\int_{C_h(p,q)} (p - q)d\mu \right|}{(\mu(C_h(p,q)) + h)|p-q|_h} \right\}.
\end{equation}
Let also $\eps(h,C_h) = \eps_{\sP}(h,C_h)$.
\end{df}

In reality, we only need $C_h$ to be defined on the model (or models, if model selection is used). However, it will be convenient to not have to keep track of several functions $C_{h,m}$.   Thus, in the following, we will always treat $C_h$ as a globally defined function. Note also that $\eps(h,C_h)$ can be made arbitrarily small by choosing $C_h$ appropriately (it is an optimization error).

\begin{df}
Let $m \subset \sP$ be a model, $h > 0$ and $C_h: \sP^2 \to \cC$ be a function. 
Let then $\tau_h(p,q) \in \{-1,1 \}$ be the sign of $\int_{C_h(p,q)} (p - q)d\mu $ and define
\begin{equation}
   t_{p,q}^{(h,C_h)} = \tau_h(p,q) \frac{P\left(C_h(p,q) \right) - \1_{C_h(p,q)}}{\mu\left(C_h(p,q) \right) + h}, 
\end{equation}
as well as the associated \emph{T-test}
\begin{equation} \label{eq_def_Ttest}
   T^{(h,C_h)}(\gX,p,q) = \frac{1}{n} \sum_{i = 1}^n t_{p,q}^{(h)}(X_i). 
\end{equation}
Set
\begin{equation} \label{eq_def_Tm}
   T_m^{(h,C_h)}(\gX,p) = \sup_{q \in m} T^{(h,C_h)}(\gX,p,q). 
\end{equation}
An $h$-approximate uniform-norm estimator (hun-estimator) associated with the class $\cC$, the model $m$, the parameter $h>0$, the set function $C_h$ and the tolerance $\delta > 0$ is, by definition, a random element $\widehat{p}_m^{(h,C_h)} \in m$ such that
\begin{equation}
   T_m^{(h,C_h)}\left(\gX,\widehat{p}_m^{(h,C_h)} \right) \leq \inf_{p \in m}  \left\{T_m^{(h,C_h)}(\gX,p) \right\} + \delta. 
\end{equation}
\end{df}

Since $m$ is countable and $\delta > 0$, $\widehat{p}_m^{(h,C_h)}$ is well-defined. $\delta > 0$ represents an optimization error and can be made arbitrarily small. To simplify, the notation $\widehat{p}_m^{(h,C_h)}$ ignores the dependence of $\widehat{p}_m^{(h)}$ on $\delta$. In the following, we shall also drop the subscript $C_h$ in the notation introduced above, whenever the definition of $C_h$ is clear from the context.

\begin{rem}
    Set $\eps = \eps_m(h,C_h)$ and
    \[ \Delta_{p\et}^{(h,C_h)}(p,q) = \tau_h(p,q) \frac{\int_{C_h(p,q)} (p - p\et) d\mu}{\mu \left(C_h(p,q) \right) + h} \]
    and note that for any non-random $p,q$,
    \[ \Delta_{p\et}^{(h,C_h)}(p,q) = \mathbb{E} \left[ T^{(h,C_h)}(\gX,p,q) \right]. \]
    Two simple facts are fundamental to the construction of our estimator:
    \begin{enumerate}
        \item For any $p,q \in m$,
        \[ |p - p\et|_h \geq \Delta_{p\et}^{(h,C_h)}(p,q) \geq (1-\eps)|p - q|_h - |q - p\et|_h \geq (1-\eps)|p - p\et|_h - (2-\eps) |q-p\et|_h. \]
        The first inequality above is obvious while the second follows immediately from the definition of $C_h(p,q)$ and the triangle inequality.
        In particular, if $p\et \in m$ (\emph{true model case}), then for any other $p \in m$,
        \[ |p - p\et|_h \geq \sup_{q \in m} \left\{  \Delta_{p\et}^{(h,C_h)}(p,q)\right\}  \geq (1-\eps)|p - p\et|_h \]
        and hence, up to $\eps = \eps_m(h,C_h)$,
        \begin{equation} \label{eq_rem_pstar-argmin}
            p\et \in \argmin_{p \in m} \sup_{q \in m} \left\{ \Delta_{p\et}^{(h,C_h)}(p,q) \right\}.
        \end{equation}
        \item Obviously, for any $p,q \in m$,
        \[ \left|T^{(h,C_h)}(\gX,p,q) - \Delta_{p \et}^{(h,C_h)}(p,q) \right| \leq \sup_{C \in \cC} \left| \frac{\sum_{i = 1}^n \1_C(X_i) - P_i\et(C)}{n(\mu(C) + h)} \right| := \sup_{C \in \cC} Z_C(h). \]
        Since the right-hand side of the above inequality is independent of $p,q,$ we get
        \begin{equation} \label{eq_rem_ubd-dev-Z}
            \sup_{p,q \in m} \left|T^{(h,C_h)}(\gX,p,q) - \Delta_{p \et}^{(h,C_h)}(p,q) \right| \leq  \sup_{C \in \cC} Z_C(h),
        \end{equation}
        whatever the model $m$. Moreover, for any given $C \in \cC$ one has in view of the definition of the $|\cdot|_h$ norm
        \begin{equation} \label{eq_rem_ubd-var-Z}
            \Var \left( Z_C(h) \right) = \frac{\sum_{i = 1}^n P_i\et(C) (1 - P_i\et(C))}{n^2 (\mu(C) + h)^2} \leq \frac{|p\et|_h}{n(\mu(C) + h)} \leq \frac{|p\et|_h}{nh}.
        \end{equation}
        Hence, for well-behaved collections $\cC$ we may expect that $\sup_{C \in \cC} Z_C(h)$ is small if $nh \to + \infty$.
    \end{enumerate}
\end{rem}

Replacing $\Delta_{p^*}^{(h,C_h)}(p,q)$ by its estimator $T^{(h,C_h)}(\gX,p,q)$ in equation \eqref{eq_rem_pstar-argmin} yields the definition of the hun-estimator (up to the tolerance $\delta$). Equation \eqref{eq_rem_ubd-var-Z} gives an estimate of the statistical error thus made. Another source of error is the approximation of the $L^\infty$ norm by $|\cdot|_h$. For $\widehat{p}_m^{(h)}$ to be a valid estimator in $L^\infty$ norm, it is necessary that $|\cdot|_h$ provide an adequate approximation to $\normfty{\cdot}$ on the model. Clearly, $|\cdot|_h$ does not uniformly approximate $\normfty{\cdot}$ on $\sP$, so this property is model-dependent: this motivates the following definition.

\begin{df}
For any model $m \subset \sP$ and any $h > 0$, let 
\begin{equation}
    \kappa_m(h) = \inf_{p,q \in m} \frac{|p - q|_h}{\normfty{p-q}}. 
\end{equation}
\end{df}

This defines a function $\kappa_m: (0,+\infty) \to [0,1]$ which can be seen to have the following properties.

\begin{lem} \label{lem_prop_kap_m}
For any model $m \subset \sbP$,
\begin{itemize}
    \item $\kappa_m$ is non-increasing
    \item If $\kappa_m(h_0) > 0$ for some $h_0 > 0$, then $\kappa_m(h) > 0$ for all $h > 0$.
    \item $\kappa_m$ is continuous, more precisely 
    \begin{equation}
        |\kappa_m(h_1) - \kappa_m(h_2)| \leq \left| 1 - \frac{h_1 \wedge h_2}{h_1 \vee h_2} \right| \kappa_m(h_1 \wedge h_2).
    \end{equation}
\end{itemize}
\end{lem}

\begin{proof}
The first property is obvious from the definition, the second is a consequence of the third. 
To prove the last inequality, note that
\begin{equation} \label{eq_lip_kap}
   |\kappa_m(h_1) - \kappa_m(h_2)| \leq \sup_{p,q \in m} \left\{ \frac{1}{\normfty{p - q}} \bigl||p - q|_{h_1} - |p - q|_{h_2} \bigr| \right\}. 
\end{equation}
Moreover, for any $f \in L^\infty$,
\begin{align*}
   \bigl||f|_{h_1} - |f|_{h_2} \bigr| &\leq \sup_{C \in \cC} \left\{ \left| \int_C f \right| \left| \frac{1}{\mu(C) + h_1} - \frac{1}{\mu(C) + h_2} \right| \right\} \\
   &\leq \sup_{C \in \cC} \left\{ \frac{|h_1 - h_2| \left| \int_C f \right|}{(\mu(C) + h_1)(\mu(C) + h_2)} \right\} \\
   &\leq |f|_{h_1 \wedge h_2} \frac{|h_1 - h_2|}{h_1 \vee h_2}.
\end{align*}
Together with equation \eqref{eq_lip_kap}, this yields the result.
\end{proof}

Moreover, if $\cC$ generates the Borel $\sigma-$algebra and $m$ is a subset of a finite-dimensional vector space, then by equivalence of norms, $\kappa_m(h) > 0$ for all $h > 0$. To bound the stochastic error of the hun-estimator, we introduce the following empirical process:
\begin{df}
For any $h > 0$, let
\[
    Z(h) = \frac{1}{n} \sup_{C \in \cC} \left\{ \frac{1}{\mu(C) + h} \left| \sum_{i = 1}^n \1_C(X_i) - P_i^*(C) \right| \right\}.
\]
\end{df}
Note that this definition does not depend on the model $m$. The risk of the hun-estimator $\hat{p}_m^{(h,C_h)}$ may be related to the constant $\kappa_m(h)$ and the process $Z(h)$ as follows.
\begin{prop} \label{prop_or_ineq_Z}
For any model $m \subset \sP$, any $h > 0$ and any $\overline{p} \in m$,
\begin{equation} \label{eq_or_ineq_Z_pbar-h}
  (1-\varepsilon_m(h,C_h)) \left|\overline{p} - \hat{p}_m^{(h,C_h)} \right|_h \leq 2 |p \et - \overline{p}|_h + 2Z(h) + \delta.   
\end{equation}
As a consequence, writing $\eps = \eps_m(h,C_h)$ for short,
\begin{equation} \label{eq_or_ineq_Z}
   (1-\varepsilon)\kappa_m(h) \times \normfty{p \et - \hat{p}_m^{(h,C_h)}} \leq \left[2 + (1-\varepsilon) \kappa_m(h) \right] \inf_{p \in m} \left\{\normfty{p\et - p} \right\} + 2Z(h) + \delta.  
\end{equation}
\end{prop}

To handle the stochastic process $Z(h)$, we state and prove a uniform Bernstein inequality. First, define the following family of events.

\begin{df} \label{def_Omega}
Let $P \et = \frac{1}{n}\sum_{i = 1}^n P_i \et$ and
\begin{equation} \label{eq_def_gam}
   \Gamma = \log(\lceil \log_2 n \rceil) + \log \left(2 \sum_{j = 0}^{V \wedge n} \binom{n}{j} \right).  
\end{equation}
For any $x > 0$, let $\Omega_x$ denote the event on which
 \begin{equation} \label{eq_bern_unif}
   \frac{1}{n} \left| \sum_{i = 1}^n \1_C(X_i) - P_i\et(C) \right| \leq \max \left(29 \sqrt{P\et(C)} \sqrt{\frac{\Gamma + x}{n}}, 20 \frac{\Gamma + x}{n} \right), 
\end{equation}
for all $C \in \cC$.
\end{df}

This class of events will govern the statistical behaviour of all procedures analyzed in this article. First, we prove the following proposition.

\begin{prop} \label{prop_bern_unif}
The event $\Omega_x$ has probability $\mathbb{P}(\Omega_x) \geq 1 - 2e^{-x}$.
\end{prop}

A result similar to Proposition \ref{prop_bern_unif} was estalished by \citet[Theorem 3]{Baraud2016} using similar methods. However, his result is stated for \emph{suprema} of empirical processes over VC-classes and in particular, the variance term in the upper bound is the supremum of the variance of the empirical process over the class. What is novel about Proposition \ref{prop_bern_unif}, to the best of our knowledge, is that it provides a \emph{pointwise} bound of the empirical process at each $C \in \cC$ in terms of the variance of the process \emph{at C}, for independent and not necessarily iid random variables. Together with Proposition \ref{prop_or_ineq_Z}, Proposition \ref{prop_bern_unif} yields the following oracle inequality for the estimator $\hat{p}_m^{(h,C_h)}$.

\begin{thm}\label{thm_or_ineq_mod}
Let $p\et = \frac{1}{n} \sum_{i = 1}^n p_i \et$ and $\Gamma$ as in equation \eqref{eq_def_gam}.
With probability greater than $1 - 2e^{-x}$, for all countable models $m \subset \sP$, all $h > 0$ and all $\overline{p} \in m$,
\begin{equation} \label{eq_or_ineq_mod_pbar-h}
(1-\varepsilon_m(h,C_h)) \left|\overline{p} - \hat{p}_m^{(h,C_h)} \right|_h \leq 2 |p \et - \overline{p}|_h + \delta + \max \left( 58 \sqrt{\frac{|p\et|_h (\Gamma + x)}{h n}}, 40 \frac{\Gamma + x}{h n}  \right).    
\end{equation}
On the same event, writing $\eps = \eps_m(h,C_h)$ for short,
\begin{equation} \label{eq_or_ineq_mod}
    \begin{split}
        (1-\varepsilon)\kappa_m(h) \times \normfty{\hat{p}_m^{(h,C_h)} - p\et} &\leq \left[2 + (1-\varepsilon) \kappa_m(h) \right] \inf_{p \in m} \normfty{p - p\et}  + \delta \\ 
&\quad + \max \left( 58 \sqrt{\frac{|p\et|_h (\Gamma + x)}{h n}}, 40 \frac{\Gamma + x}{h n}  \right). 
    \end{split}
\end{equation}
\end{thm}

The quality of the bound provided by Theorem $1$ depends on the constant $\kappa_m(h)$. If $h_m$ is such that $\kappa_m(h_m) \geq \frac{1}{2}$ (say), then Theorem $1$ yields a "true" oracle inequality, with remainder term of order 
\[ \sqrt{\frac{|p\et|_{h_m} \Gamma}{n h_m}} \leq \sqrt{\frac{\normfty{p\et} \Gamma}{n h_m}}. \] 
Clearly, this value of $h_m$ depends strongly on the class $\cC$ and the model $m$. Later, we will show that $m,\cC$ can be chosen such that equation \eqref{eq_or_ineq_mod} yields the minimax convergence rate over classes of smooth functions. More generally, one can ask when a constant $h_m$ even exists. A sufficient condition for this is to have $\kappa_m(h) \to 1$ as $h \to 0$.

Existence of $h_m$ provides an oracle inequality with a fixed constant ($5$, say) in front of the approximation error and a remainder term of order $\mathcal{O}\left( 1/\sqrt{n} \right)$ as $n \to +\infty$ for a fixed model $m$. This is the expected rate of convergence for finite-dimensional models. We now show that, on $\mathbb{R}^d$ with Lebesgue measure $\mu$, there are universal classes of sets $\cC$ such that $\kappa_m(h) \to 1$ holds for all finite-dimensional vector space models $m$ over which $\normfty{\cdot}$ is a norm.

\begin{prop} \label{prop_consistence}
Assume that $\cC$ contains a sub-collection $\cC_0$ satisfying the following conditions:
\begin{itemize}
    \item For all $\delta > 0$, $\bigcup_{C \in \cC_{0,\delta}} \overline{C} = \mathbb{R}^d$, where
    \begin{equation}
       \cC_{0,\delta} = \left\{ C \in \cC_0 : \mathrm{diam}(C) \leq \delta \right\} 
    \end{equation}
    \item $\inf \left\{ \frac{\mu(C)}{\mathrm{diam}(C)^d} : C \in \cC_0 \right\} > 0$,
\end{itemize}
where $\mathrm{diam}(C)$ denotes the diameter of $C$. Then for any $f \in \sbP$, $\lim_{h \to 0} |f|_h = \normfty{f}$. As a consequence, for any finite-dimensional vector sub-space $m$ of $\sP$, $\lim_{h \to 0} \left\{ \kappa_{m}(h) \right\} = 1$. 
\end{prop}
Classes of sets $\cC$ which satisfy the assumptions of Proposition \ref{prop_consistence} while having finite VC dimension include simplices, "box sets" (products of intervals), dyadic cubes, euclidean balls, ellipsoids, and many more.

\subsection{Piecewise polynomials} \label{sec-piec_poly-1mod}
To obtain more quantitative results about the constant $\kappa_m(h)$, it is necessary
to look at specific classes of models. Here, we restrict attention to classes of piecewise polynomial functions on partitions of $\mathbb{R}^d$, because these classes are simple to define and have optimal approximation properties. However, we are confident that similar results could be proved for other classical function spaces, such as wavelet spaces or trigonometric polynomials. In the rest of this section, we shall assume that $\mu$ is the Lebesgue measure on $\mathbb{R}^d$.

First, it is necessary to introduce some definitions and notations. A (multivariate) \emph{polynomial function} on $\mathbb{R}^d$ is a function of the form:
\begin{equation}
   f: x \mapsto \sum_{a \in \cA} c(a) \prod_{i = 1}^d x_i^{a(i)},  
\end{equation}
where $\cA$ is a finite set of functions $a: \{1,\ldots,d\} \to \mathbb{N}$ and $c: \cA \to \mathbb{R}$ is a function. 
Its \emph{degree} (in the usual sense) is defined to be
\begin{equation}
   \deg(f) = \max_{a \in \cA} \sum_{i = 1}^d a(i). 
\end{equation}
It satisfies the usual relations, $\deg(f g) = \deg(f) + \deg(g)$ and $\deg(f + g) \leq \max(\deg(f),\deg(g))$.
We define also the \emph{directional degree} in direction $i \in \{1,\ldots,d\}$ to be
\begin{equation}
    \deg_i(f) = \max_{a \in \cA} a(i),
\end{equation}
which satisfies the same relations. 

Let $\pol{\infty}{d}$ be the space of all multivariate polynomial functions on $\mathbb{R}^d$.
We define the following two families of spaces of polynomials with bounded degrees: first, given 
$r \in \mathbb{N}$, let
\begin{equation}
   \pol{r}{d} = \{ f \in \pol{\infty}{d} : \deg(f) \leq r \}.  
\end{equation}
Secondly, for all vectors $\gr \in \mathbb{N}^d$, let
\begin{equation}
   \poldir{\gr}{d} = \{f \in \pol{\infty}{d} : \forall i \in \{1,\ldots,d \}, \deg_i(f) \leq r_i \}.
\end{equation}
The two families of spaces are related by the following inclusions:
\[ \poldir{\gr}{d} \subset \pol{\norm{\gr}_1}{d} \subset \poldir{\norm{\gr}_1 \mathbf{1}}{d}, \]
where $\mathbf{1}$ is the "all-one" vector, $\mathbf{1} = (1,\ldots,1)$. We can now define models of piecewise polynomial functions. 
\begin{df}
Given a finite or countable and measurable partition $\cI$ of $\mathbb{R}^d$ and $r \in \mathbb{N}$, let
$m(r,\cI)$ denote the set of functions of the form
\[f = \sum_{I \in \cI} f_I \1_I, \]
where for each $I \in \cI$, $f_I \in \pol{r}{d}$ is a polynomial with rational coefficients
and the set $\{ I \in \cI: f_I \neq 0 \}$ is finite. Let $\bar{m}(r,\cI) = \overline{m(r,\cI)}$, the closure of $m(r,\cI)$ in $L^\infty\left( \mathbb{R}^d \right)$. Given $\gr \in \mathbb{N}^d$, let $m_{dir}(\gr,\cI), \bar{m}_{dir}(\gr,d)$ be defined similarly,
with $\poldir{\gr}{d}$ instead of $\pol{r}{d}$.
\end{df}

Let the model $m = m(r,\cI)$ for some partition $\cI$.
If $\cI$ is finite, then $m$ is finite-dimensional and the previous proposition applies. In general, to establish an explicit lower bound on $\kappa_m(h)$, we require the partition $\cI$ to satisfy the following three conditions.

\begin{ass}\label{ass_part}
\text{ }
\begin{itemize}
    \item $\cC$ contains translated and scaled copies of the interior $\mathring{I}$ of any $I \in \cI$, i.e
    \begin{equation} \label{hyp_comp_C_I}
    \left\{ x + \lambda \mathring{I} : I \in \cI, x \in \mathbb{R}^d, \lambda > 0 \right\} \subset \cC.    
    \end{equation}
\item There is a lower bound on the volume of the elements of $\cI$:
\begin{equation} \label{hyp_lbd_mu_part}
   h_0 := \min_{I \in \cI} \mu(I) > 0.  
\end{equation}
\item The elements of $\cI$ are bounded convex sets.
\end{itemize}
\end{ass}

Under Assumption \ref{ass_part}, for any $f = p - q \in m$, an appropriate set $C_{h,m}(f) \in \cC$ can be constructed as follows. Since the collection $(f \1_I)_{I \in \cI}$ has finite support, the supremum $\sup_{I \in \cI} \normfty{f \1_I}$ is reached at some $I_*(f)$. Let $\mathring{I}_*(f)$ denote the topological interior of $I_*(f)$. $f$ coincides on $I_*(f)$ with a polynomial $f_*$, and $|f_*|$ reaches its maximum on $\overline{I_*(f)}$ at some $x_*(f)$.
Let finally
\begin{equation} \label{eq_def_C_poly}
    C_{h,m}(f) = (1 - \theta_{m}(h)) x_*(f) + \theta_{m}(h) \mathring{I}_*(f),
\end{equation}
where $\theta_{m}(h) \in (0,1)$ is a function given by equation \eqref{eq_def_thetam} below. By Assumption \ref{ass_part}, $C_{h,m}(f) \in \cC$ and by convexity of $I_*(f)$, $C_{h,m}(f) \subset \mathring{I}_*(f)$.
The following lower bound holds.
\begin{prop} \label{prop_lb_kap_m_poly}
For all $u > 0$, let
\begin{equation} \label{eq_def_gamma_rd}
   \gamma_{r,d}(u) = \max \left( \frac{1}{2(d+1)} \left[\frac{u^{-1}}{(2r^2)^d} \wedge 1\right], \left[1 - (2r^2)^{\frac{d}{d+1}} u^{\frac{1}{d+1}} \right]_+^2  \right).
\end{equation}
Assume that Assumption \ref{ass_part} holds. Let then
\begin{equation} \label{eq_def_thetam}
    \theta_{m}(h) = \begin{cases}
    &\frac{d}{d+1} \frac{1}{2r^2} \text{ if } \gamma_{r,d} \left(\frac{h}{h_0} \right) = \frac{1}{2(d+1)} \left[\frac{h_0}{(2r^2)^{d} h} \wedge 1\right] \\
    & \left( \frac{h}{2r^2 h_0} \right)^{\frac{1}{d+1}} \text{ otherwise}.
    \end{cases}
\end{equation}
For all $f \in m = m(r,\cI)$,
\[ \frac{\left| \int_{C_{h,m}(f)} f d\mu \right|}{\mu \left( C_{h,m}(f)  \right) + h} \geq \gamma_{r,d} \left(\frac{h}{h_0} \right) \normfty{f}. \]
In particular, since $m$ is a $\Q-$vector space, $\kappa_m(h) \geq \gamma_{r,d} \left(\frac{h}{h_0} \right) $.
\end{prop}

In particular, $\kappa_m(h)$ converges to $1$  as $(h/h_0) \to 0$, and the rate of convergence depends only on the dimension $d$ (and not on the partition $\cI$). Thus, the estimation error behaves essentially like $1/\sqrt{h_0}$, when $h$ is well chosen. For concreteness, consider the case of the collection $\cC$ of cartesian products of $d$ intervals, with $\cI \subset \cC$ a partition of $\mathbb{R}^d$. Theorem \ref{thm_or_ineq_mod} and Proposition \ref{prop_approx_poly} yield the following corollary. 

\begin{cor} \label{cor_mod_poly}
Let $m = m(r,\cI)$, $\cI \subset \cC$ satisfying \ref{ass_part}, $\cC$ the collection of cartesian products of $d$ intervals,
\begin{equation} \label{eq_def_hm1}
    h_m = \frac{\left(1 - \frac{1}{\sqrt{2}} \right)^{d+1}}{(2r^2)^d} h_0. 
\end{equation}
Let $\hat{p}_m^{(h_m)}$ be the hun-estimator based on the sets $C_{h_m,m}(p-q)$ ($p,q \in m$) defined above. Then
\begin{align*}
    \mathbb{E}\left[ \normfty{\hat{p}_m^{(h_m)} - p\et} \right] &\leq 5 \min_{p \in m} \{\normfty{p - p\et} \} + 274 (2d+1)(3r)^{2d} \frac{\log(en)}{h_0 n} + 2\delta \\
    &\quad + 215 \sqrt{2d+1} \min \left( (3r)^d \sqrt{\normfty{p\et}} \sqrt{\frac{\log(en)}{h_0 n}}, (3r)^{2d} \frac{\sqrt{\log(en)}}{h_0 \sqrt{n}} \right).
\end{align*}
\end{cor}

The remainder term in the oracle inequality above is equivalent to
\[ c_{r,d} \sqrt{\normfty{p\et}}  \sqrt{\frac{\log(n)}{h_0 n}} \]
for some constant $c_{r,d}$ (depending on $r,d$ only), in the asymptotic regime where $h_0 \to 0$ and $h_0 n \to +\infty$. We show below that this is optimal for sufficiently "regular" partitions.  Though we do not believe that the constant $c_{r,d}$ is optimal, exponential behaviour is expected since
\[ \dim \left( \cP_{r,d} \right) = {r+d \choose d} \geq \min \left(e^{\frac{r}{2}}, \frac{4^d}{2d+1} \right). \]

To assess the optimality of the remainder term $\sqrt{\normfty{p\et}}  \sqrt{\frac{\log(n)}{h_0 n}}$ of Corollary \ref{cor_mod_poly} and more generally of Theorem \ref{thm_or_ineq_mod}, we prove a minimax lower bound on the class 
\begin{equation} \label{eq_def_mLI}
   m_L(0,\cI) = \left\{ \sum_{I \in \cI} c_I \1_I : c \in [0,L]^{\cI}, \sum_{I \in \cI} c_I \mu(I) = 1 \right\}  
\end{equation}
of pdfs which are piecewise constant on the blocks of the partition $\cI$ and uniformly bounded by $L > 0$.

Note that the set $m_L(0,\cI)$ may be empty (if $\cI$ does not contain blocks of finite measure), or a singleton (if $\cI$ has exactly one block of finite measure). If $\cI$ has a finite number of blocks of finite measure, then $m_L(0,\cI)$ will also be empty if $L$ is too small. In such cases, estimation on $m_L(0,\cI)$ is trivial. In general, the following minimax lower bound holds.

\begin{thm} \label{thm_minimax_lbd_hist}
Let $(\cX,\cB,\mu)$ be a $\sigma-$finite measure space and $\cI$ a countable, measurable partition of $\cX$ into blocks of positive measure. Let
\[ \cX_0 = \bigcup \{ I \in \cI : \mu(I) < +\infty \}. \]
For any $h > 0$, let
\[ M(h) = \left| \left\{ I \in \cI : \mu(I) \leq h \right\} \right|. \]
For any $L > 0$ and $n \geq 1$, define $\psi_n(\cI,L) > 0$ by 
\begin{equation} \label{eq_def_psi-n}
    \psi_n(\cI,L)^2 = \sup_{h > 0} 
\left\{ \frac{L}{h n} \log \left(1 + \min \Bigl( M(h), \left\lfloor \frac{1}{L h} \right\rfloor \Bigr) \right) \right\}.
\end{equation}
Then, for any $\theta \in \left( \frac{1}{2}, 1 \right)$ and any $L \geq \frac{1}{\theta \mu(\cX_0)}$,
\[ \inf_{\hat{p}} \sup_{p\et \in m_L(0,\cI)} \mathbb{E} \left[ \normfty{\hat{p} - p\et} \right] \geq \frac{1}{40} \min \left((1-\theta)L, \sqrt{\theta(1 - \theta)} \psi_n(\cI,L) \right),  \]
where the infimum runs over all estimators $\hat{p}$ of $p\et$, based on an iid sample of size $n$ drawn from $p\et$.
\end{thm}

Though the class $m_L(0,\cI)$ is a simple one, and the proof of Theorem \ref{thm_minimax_lbd_hist} is classical, Theorem \ref{thm_minimax_lbd_hist}  is, to the best of our knowledge, the first minimax lower bound for histograms in density estimation in sup-norm. The lower bound involves the parameters $L,h,n$ and $\theta$, as well as the function $M$ which depends on the partition $\cI$.

The parameter $\theta$ reflects the fact that if $L$ is too small, then the model is empty, and if $\mu(\cX_0) L = 1$, then the model contains precisely one element (the uniform distribution on $\cX_0$). As soon as $L$ is greater than this minimum value by a  constant factor $\frac{1}{\theta}$, the lower bound is of order
$ \min(L, \psi_n(\cI,L)). $ The minimum with $L$ reflects the fact that we can always use any fixed $p_0 \in m_L(0,\cI)$ as an estimator, which has risk bounded by $L$. As soon as $n$ is large enough, such that this trivial estimator is sub-optimal, the minimax risk becomes proportional to $\psi_n(\cI,L)$. 

This term, $\psi_n(\cI,L)$, is somewhat complicated. For the purpose of this discussion, fix a partition $\cI$ and let
$ h_0 = \inf_{I \in \cI} \mu(I)$ (as in equation \eqref{hyp_lbd_mu_part} of Assumption \ref{ass_part}).
If $h_0 L < 1$, then for any $h \in \left(h_0, 1/L \right]$, $M(h) \geq 1$ and $Lh \leq 1$, which implies that 
\[\psi_n(\cI,L) \geq \sqrt{\frac{L \log 2}{h n}}. \]
On the other hand, if $h_0 L \geq 1$, then since the models $(m_t(0,\cI))_{t > 0}$ are nested, the minimax risk on $m_L(0,\cI)$ is greater than the minimax risk on $m_{\frac{1}{h}}(0,\cI)$ for any $h > h_0$. This yields the following corollary.

\begin{cor} \label{cor_minimax_lbd}
Let $\cI$ be a countable partition of $\cX$ into blocks of finite, positive measure. For any $L \geq \frac{2}{\mu(\cX)},$
\[ \inf_{\hat{p}} \sup_{p \in m_L(0,\cI)} \mathbb{E} \left[ \normfty{\hat{p} - p} \right] \geq \frac{1}{80} \min \left(L, \sqrt{\frac{L \log 2}{h_0 n}}, \frac{\sqrt{\log 2}}{h_0 \sqrt{n}} \right), \]
where $h_0 = \inf_{I \in \cI} \mu(I)$.
\end{cor}

Comparing Corollary \ref{cor_minimax_lbd} with the minimax upper bound resulting from Corollary \ref{cor_mod_poly}, we see that Corollary \ref{cor_mod_poly} is optimal, possibly up to $\log n$ factors and the remainder term $1/(h_0n)$, which is negligible relative to the minimax lower bound whenever $\sqrt{\frac{L}{h_0 n}} \ll L$, i.e whenever a non-trivial estimator is required.

If we assume additionally that 
\[M(2h_0) = \left| \left\{ I \in \cI : h_0 \leq \mu(I) \leq 2h_0  \right\} \right| \geq n^\alpha \]  
and that $2 L n^\alpha h_0 \leq 1$ for some fixed $\alpha \in (0,1)$, then by equation \eqref{eq_def_psi-n},
\[ \psi_n(\cI,L) \geq \sqrt{\frac{L}{2 h_0 n}} \sqrt{\log (1 + \lfloor n^\alpha \rfloor)} \geq \sqrt{\frac{\alpha L \log n}{2 h_0 n}},  \]
in which case the upper bound of Corollary \ref{cor_mod_poly} is optimal up to a constant depending only on $\alpha$. 

Moreover, if $\cI_n$ are regular partitions of $\R^d$ into blocks of volume $h_n$, where $\limsup_{n \to +\infty} \{ n^\alpha h_n \} < +\infty$, then 
\[ \liminf_{n \to +\infty} \left\{ \psi_n(\cI_n,L) \times \sqrt{\frac{h_n n}{L \log n}} \right\} \geq \sqrt{\alpha}, \]
which proves the asymptotic optimality of Corollary \ref{cor_mod_poly} in this non-parametric setting.

\subsection{Computational aspects}
The freedom in the choice of model $m$, class $\cC,$ sets $C_h(p,q)$ and parameters $h,\delta$ and $\eps$ gives considerable flexibility in the construction of the hun-estimator. Depending on the choice of these quantities, computation of the hun-estimator may or may not be feasible in practice. 

In this section, we give an example where computation of the hun-estimator can be achieved with reasonable complexity.  Consider the piecewise polynomial model $m(r,\cI(j))$ with $\cI(j)$ a cubic dyadic partition of $\R^d$:
\[ \cI(j) = \left\{ \prod_{i = 1}^d  \left[k_i 2^{-j}, (k_i+1)2^{-j}  \right) : \gk \in \Z^d \right\}. \]
Let $\cC = \cI(j+l)$ for some large enough positive integer $l$. For any $p,q \in m(r,\cI(j)),$ let $C_h(p,q)$ be a cell of $\cC$ wherein $|p-q|$ reaches its maximum (extending $p-q$ by continuity to the boundary if necessary). Then, by an argument similar to that of Section \ref{sec-piec_poly-1mod} and Proposition \ref{prop_lb_kap_m_poly}, we can show that
\[ \frac{\left| \int_{C_h(p,q)} (p-q)d\mu \right|}{\mu(C_h(p,q)) + h} \geq \frac{1}{2} \normfty{p-q}, \]
for all $h \leq c 2^{-(j+l)},$ where $c$ is a constant depending only on $r,d$. 

Since the model $m = m(r,\cI(j))$ is a vector space, for any fixed $p$, $\{p-q : q \in m\} = m$. Since a piecewise polynomial of degree $r \geq 2$ can reach its essential supremum anywhere and be of either sign at this point, $\{C_h(p,q): q \in m\} = \cC$ and
\[ \sup_{q \in m} T^{(h)}(\gX,p,q) = \sup_{C \in \cC} \frac{1}{2^{-(k+l)d} + h} \left| \frac{1}{n} \sum_{i = 1}^n \1_C(X_i) - P(C) \right|. \]
The hun-estimator is obtained by optimizing the above quantity over $p \in m$. Since the dyadic partitions are nested, the polynomial pieces can be chosen independently, by optimizing the function
\[ F_I(p_I) = \max_{C \in \cC: C \subset I} \left|N_n(C) - n \int_{C} p_I d\mu \right| \]
over polynomials $p_I$ of degree $\leq r$ on $I \in \cI(j)$, where $N_n(C)$ denotes the number of data points in $C$. For each $I \in \cI(j)$ which contains no data point, we can clearly set $p_I = 0$ on $I$. Hence, for each of the at most $n$ cells of $\cI(j)$ which contain a data point, we have to minimize the convex function $F_I$ over the vector space of polynomial functions of degree $\leq r$. Note that  the maximum in the definition of $F_I$ runs over a maximum of $2^{ld}$ sets. 

Thus, this hun-estimator can be computed in two steps:
\begin{itemize}
    \item For each cell $C \in \cI(j+l)$ which contains some $X_i$, count the number of data points in it.
    \item For each non-empty $I \in \cI(j),$ optimize the convex function $F_I$ over the set of polynomial functions of degree at most $r$.
\end{itemize}
It follows that this hun-estimator can be computed in polynomial time with respect to the sample size $n$. The computational complexity is exponential in the data dimension $d$, which is logical given that the model dimension is itself exponential in $d$ ("curse of dimensionality").    

\section{Structural assumptions and dimension reduction}
Nonparametric estimation under classical smoothness assumptions is known to suffer from a \emph{curse of dimensionality}: the optimal rates of convergence decrease with the dimension, which restricts applicability of such methods to low dimensions. 
To mitigate this problem, a number of dimensionally reduced models have been put forward in the literature.
This dimension reduction is accomplished by making a \emph{structural assumption} on the density $p$, in order to reduce the number of arguments to the functional parameters. 

\subsection{Some models based on structural assumptions}
We will consider three main examples.

\begin{ex}[Single and multi-index models]
  This model, proposed by  \citet{SamTsyb2007}, was named "multi-index model" by analogy with the non-parametric regression setting. 
It consists of densities of the form
\[ p(x) = f(B^T x) \phi_d(x), \]
where $\phi_d$ is the $d-$dimensional Gaussian density function, $B$ is a $d\times k$ matrix with orthogonal columns ($k < d$) and $f: \mathbb{R}^k \to +\infty$ is the functional parameter, on which additional smoothness assumptions may be made. In particular, we shall be interested in the "single-index" case ($k = 1$):
\[ p(x) = f(\langle \theta,x \rangle) \phi_d(x), \]
for some unit vector $\theta$.
\end{ex}

\begin{ex}[Independence structure]
In this model, proposed by \citet{Lepski2013} we assume that $p$ can be written as a product:
\[ p(x) = \prod_{B \in \cB} p_B(x_B), \]
where $\cB$ is a partition of $\{1,\ldots,d\}$, $x_B = (x_i)_{i \in B}$ and the $p_B : \R^B \to \R_+$ are probability density functions satisfying certain smoothness assumptions.    
\end{ex}

\begin{ex}[Independent component analysis (ICA)]
    This model was proposed by \citet{SamTsyb2004} and studied more recently by \citet{LepReb2020}. It consists of densities of the form
    \[ p(x) = |\mathrm{det} B| \prod_{j = 1}^d p_j(\beta_j^T x), \]
    where the $p_j$ are univariate probability densities and $B$ is a non-singular matrix with columns $(\beta_j)_{1 \leq j \leq d}$.
\end{ex}

To the best of our knowledge, estimators with optimal performance with respect to the supremum norm have not been constructed for the single/multi-index or ICA models.
For the independence structure assumption, optimal estimators were defined by \citet{Lepski2013}.

We believe that the hun-estimator also has optimal performance in these settings. We shall prove this for the single-index model, and give justification for the independence structure assumption. However, the analysis is more involved than in the previous section. In particular, because of the non-linearity of these models, the value of the constant $\kappa_m(h)$ may not reflect the expected performance. Instead, we resort to equation \eqref{eq_or_ineq_mod_pbar-h} of Theorem \ref{thm_or_ineq_mod} combined with a direct, possibly non-linear lower bound on $|\overline{p} - p|_h$ for a fixed $\overline{p} \in m$ and all $p \in m$. 

\subsection{Non-linear bounds and the projection argument illustrated under the independence structure assumption} \label{sec_proj_arg}

To illustrate the flexibility of our approach, consider the independence structure assumption with \emph{known} partition $\cB$, for a density $p$ defined on $[0,1]^d$. The structure of the density suggests the following model:
\[ m_{\cB} = \left\{ \prod_{B \in \cB} p_B(x_B) : p_B \in m_{B}, p_B \text{ pdf} \right\}, \]
where for any set of indices $B,$ $m_B$ is a finite-dimensional space of functions $f: [0,1]^B \to \R$, for example the piecewise polynomials defined in the previous section. Let $\cC$ be the class of products of $d$ intervals. Let $\cC_B$ be the corresponding class in $\R^B$. For any $C \in \cC_B,$ let 
\[\Tilde{C} = \left\{ x \in [0,1]^d : x_B = (x_i)_{i \in B} \in C \right\} \in \cC. \]
Then for any $p,q \in m_{\cB}$,
\[ \frac{\int_{\Tilde{C}} (p-q) d\mu}{\mu(\Tilde{C}) + h} = \frac{\int_{C} (p_B - q_B)(x_B) d\mu(x_B)}{\mu(C) + h},   \]
which proves that
\begin{equation} \label{eq_ubd_hnorm_indt_struct}
    |p-q|_h \geq \max_{B \in \cB} |p_B - q_B|_h \geq \left(\min_{B \in \cB} \kappa_{m_{B}}(h) \right) \times \max_{B \in \cB} \normfty{p_B - q_B}. 
\end{equation}
This relates the behaviour of $|\cdot|_h$ on $m_{\cB}$ to that on the lower dimensional spaces $m_{B}$.
It remains to relate $\normfty{p - q}$ to $\max_{B \in \cB} \normfty{p_B - q_B}$. The following simple non-linear bound holds.
\begin{lem} \label{lem_ind_struct}
  For any $\overline{p}, p \in m_{\cB},$
  \[ \normfty{\overline{p} - p} \leq \normfty{\overline{p}}\left[ \left(1 + \max_{B \in \cB} \normfty{\overline{p}_B - p_B} \right)^{|\cB|} - 1 \right]. \]
\end{lem}

Plugging this bound and equation \eqref{eq_ubd_hnorm_indt_struct} into equation \eqref{eq_or_ineq_mod_pbar-h} of Theorem \ref{thm_or_ineq_mod} yields an oracle inequality for the hun-estimator on $m_{\cB}$. Of course, in practice we do not want to assume that the independence structure $\cB$ is known. This is a common feature of the above models: given some finite-dimensional parameter (a partition, multi-index, or matrix), estimation is easy by standard non-parametric methods, and the main difficulty lies in the simultaneous estimation of the parameter and one or more functions.

In the context of hun-estimation, this difficulty can be handled, at least in some cases, by a \emph{projection argument}. Consider a model of the form
\[m = \bigcup \left\{m_\theta : \theta \in \Theta \right\}, \]
where $\theta$ is the unknown parameter and for each $\theta$, a lower bound is known on $|p-q|_h$ for $p,q \in m_\theta$ (for example, a good value of $\kappa_{m_\theta}(h)$). 

Assume that for each $\theta \in \Theta$, we have projections $Q_\theta: m \to m_\theta$ which are uniformly Lipschitz, or at least equi-continuous.  For example, for an independence structure with partition $\cB$, we may define
\[ Q_{\cB}(p)(x)  = \prod_{B \in \cB} p_B(x_B) = \prod_{B \in \cB} \int_{[0,1]^{B^c}} p(x) dx_{B^c},  \]
the product of the appropriate marginals. Then, for any $\overline{p} \in m$ and $p \in m_\theta$, we have for example
\begin{align*}
    \normfty{Q_\theta \overline{p} - p} &\leq \frac{|Q_\theta \overline{p} - p|_h}{\kappa_{m_\theta}(h)} \\
    &\leq \frac{w(\overline{p},|\overline{p} - p|_h)}{\kappa_{m_\theta}(h)},
\end{align*}
where $w$ is a modulus of continuity of the $Q_\theta$ at $\overline{p}$. 

On the other hand, for any $q\in m_\theta$,
\[ \normfty{Q_\theta \overline{p} - \overline{p}} \leq \normfty{Q_\theta \overline{p} - Q_\theta q} + \normfty{q - \overline{p}} \leq w(\overline{p}, \normfty{\overline{p} - q}) + \normfty{q - \overline{p}}. \]
Thus, if we can show that for all $\theta \in \Theta, p \in m_\theta$, $|\overline{p} - p|_h$ bounds the approximation error 
\[ \inf_{q \in m_\theta} \normfty{\overline{p} - q}, \]
then the triangle inequality will yield the desired bound of $\normfty{\overline{p} - p}$ by a function of $|\overline{p} - p|_h$ and $\overline{p}$.

This last step requires a suitable choice of the $\cC$ and $C \in \cC$. For example, in the independence structure setting, we remark that for any partitions $\cB_0,\cB_1$, any $B \in \cB_0, C_B \in \cC_B$ and $p \in m_{\cB_1}$,
\[ \int_{\Tilde{C}_B} p d\mu = \int_{\Tilde{C}_B} \Tilde{p} d\mu \text{ where } \Tilde{p} = \prod_{B_0 \in \cB_0} \prod_{B_1 \in \cB_1} p_{B_0 \cap B_1}, \]
where as before, for a set of variables $A,$ $p_A$ denotes the marginal of $p$ along the variables in $A$. For appropriate models $(m_A)_{A \subset \{1,\ldots,d\}}$, $\Tilde{p}$ belongs to $m_{\cB_0} \cap m_{\cB_1}$. This implies that, when $\overline{p} \in m_{\cB_0}$ and $p \in m_{\cB_1}$,
\[ \left| \overline{p} - p \right|_h \geq \min_{B \in \cB_0} \kappa_{m_B}(h) \times \max_{B \in \cB_0} \normfty{\overline{p} - \Tilde{p}},  \]
which bounds $\normfty{\overline{p} - \Tilde{p}}$ by Lemma \ref{lem_ind_struct}. Since also $\Tilde{p} \in m_{\cB_1}$, this is greater than the approximation error of $\overline{p}$ by $m_{\cB_1},$ as required.

\subsection{Single and multi-index models} \label{sec_si_mod}
For these models, we conjecture  that the hun-estimator attains the optimal rate and that this rate is the same as for estimating the functional parameter on its domain $\R^k,$ $k$ being the number of indices. For the sake of simplicity, we only proved the result in the single-index case, but the technique used should also apply to the general multi-index setting. 

Let $\alpha > 0$ be the smoothness parameter. Given a function $f$ on $\R^k$, let
\[ \norm{f}_{\alpha,\infty} = \max \left( \max_{0 \leq k \leq \lceil \alpha \rceil - 1} \normfty{f^{(k)}}, H_{\alpha + 1 - \lceil \alpha \rceil}\left( f^{(\lceil \alpha \rceil - 1)} \right) \right), \]
where $H_\theta(\cdot)$ denotes the Hölder semi-norm of order $\theta$, for any $\theta \in (0,1]$ (in particular, $H_1(\cdot)$ is the Lipschitz semi-norm). Define the \emph{$\alpha,L$-smooth single index model} on a euclidean space $E$ as  
\[ \cF_{\alpha,L}^{si}(E) = \left\{ p: x \mapsto f \left(\langle \theta,x \rangle \right) e^{\frac{\langle \theta,x\rangle^2}{2}} e^{- \frac{\norm{x}^2}{2}}: \theta \in E, \norm{f}_{\alpha,\infty} \leq L \right\}. \]
Clearly, estimating $p$ as above (with $\theta$ unknown) is at least as hard as estimating the univariate pdf $f$. Using hun-estimation, we can show that it is no harder, at least in terms of rates of convergence.

To estimate densities belonging to $\cF_{\alpha,L}^{si}(E)$, consider for any $r \in \N, h_0 > 0$ the spline space
\[ \cS(r,h_0) = m(r,\cI(h_0)) \cap C^{r-1}(\R) \text{ where } \cI(h_0) = \left\{ [ih_0,(i+1)h_0) : i \in \Z \right\} \]
and let 
\[m_r^{si}(E,h_0) = \left\{ x \mapsto f(\langle \theta,x \rangle) e^{\frac{\langle \theta,x\rangle^2}{2}} e^{- \frac{\norm{x}^2}{2}} : \theta \in E_0, f \in \cS(r,h_0) \right\},  \]
where $E_0$ is a countable dense subset of $E$.
Let $\cC$ be the class of triangular prisms, that is to say, the class of sets of the form:
\[ \left\{ \sum_{i = 1}^d x_i e_i : (x_1,x_2) \in Tr, x_i \in (a_i,b_i), 3 \leq i \leq d \right\}, \]
where $(e_i)_{1 \leq i \leq d}$ is an orthonormal basis of $E$ and $Tr \subset \mathbb{R}^2$ is a triangle. Theorem \ref{thm_si_rates} below shows that the hun-estimator on $m_r^{si}(h_0)$ attains the minimax rate on $\cF_{\alpha,L}^{si}(E)$ 
when $h_0$ is suitably chosen.

\begin{thm} \label{thm_si_rates}
    Let $\alpha > 0, L > 0$ and let $r \geq \alpha$ be an integer. Let
    \[ h_0 = \left( \frac{\log n}{L n} \right)^{\frac{1}{2\alpha + 1}}. \]
    There exist constants $c_1,c_2,c_3$ depending only on $\alpha, r, \mathrm{ dim} \ E$ such that, when $m = m_r^{si}(h_0)$ and $h = c_1 h_0$, given $C_h: \sP^2 \to \cC$ such that $\eps(h,C_h) \leq \frac{1}{97}$ and sufficiently small $\delta > 0$, the hun estimator satisfies the following oracle inequality:
    \[ \mathbb{E} \left[ \normfty{\hat{p}_m^{(h,C_h)} - p\et } \right] \leq 100 \inf_{\overline{p} \in \cF_{\alpha,L}^{si}(E)} \normfty{\overline{p} - p\et} + c_2 L^{\frac{\alpha + 1}{2\alpha + 1}} \left( \frac{\log n}{n} \right)^{\frac{\alpha}{2\alpha + 1}} + c_3 L^{\frac{1}{2\alpha + 1}} \left( \frac{\log n}{n} \right)^{\frac{2\alpha}{2\alpha + 1}}. \]
\end{thm}

Note that, even though the hun estimator attains the minimax rate on $\cF_{\alpha,L}^{si}(E),$ there does not seem to be an oracle inequality on the model $m_r^{si}(E,h_0)$ comparable to Corollary \ref{cor_mod_poly}. This is due to the uncertainty on the index $\theta$ creating uncertainty about the location of the spline knots. The fact that these knots are known and fixed is crucial to the proof of Proposition \ref{prop_lb_kap_m_poly}. For similar reasons, the proof requires the elements of the model $m$ to be as regular as the functions to be estimated, which is why this section uses splines instead of piecewise polynomials.  

\section{Model selection and adaptivity}
\subsection{General approach} \label{sec:mod_select_method}

Let $\mathcal{M}$ be a collection of \emph{models} and let $\gM= \cup_{m \in \mathcal{M}} m$. Let $C_h: \sP^2 \to \cC$ be a collection of functions (parametrized by $h$) such that $\varepsilon_{\gM}(h,C_h) \leq \varepsilon < 1$.  In principle, the tests $t_{p,q}^{(h,C_h)}$ for a fixed $h$ could be used to select an element of $\gM$. However, the corresponding constant
$\kappa_\gM(h) \leq \inf_{m \in \cM} \kappa_m(h)$
is worse than that of all the individual models. Equivalently, if we let, for any model $m \subset \sP$, $h_m$ be such that $\kappa_m(h_m) = \frac{1}{2}$ (say), then 
$ h_{\gM} \leq \inf_{m \in \cM} h_m. $
In particular, if the models are nested, the value of $h_{\gM}$ is that of the largest model.  

It would be desirable to instead use different values of $h$ depending on the models to which $p,q$ belong, so as to obtain an estimator which performs as well as the best single-model estimator in the collection $(\hat{p}_m^{(h_m,C_{h_m})})_{m \in \cM}$. To achieve this goal of \emph{model selection}, some means of estimating the statistical error $Z(h)$ is needed. Theorem \ref{thm_or_ineq_mod} provides an upper bound on $Z(h)$ which is almost fully explicit: it only depends on $\gP\et$ through $|p \et|_h$. We now show how this quantity can be estimated.

\begin{df}
For any $h > 0$, let
\[ |\hat{p}|_h = \sup_{C \in \cC} \left\{ \frac{\sum_{i = 1}^n \1_C(X_i)}{n(\mu(C) + h)} \right\}.   \]
\end{df}

The following proposition shows that $|\hat{p}|_h$ is an adequate estimator of $|p\et|_h$.

\begin{prop} \label{prop_comp_p-eth_p-hath}
On $\Omega_x$, for all $\theta \in (0,2)$,
\begin{align}
    |p\et|_h &\leq \frac{1}{1 - \frac{\theta}{2}} |\hat{p}|_h + \frac{29^2}{\theta(2 - \theta)} \frac{\Gamma + x}{h n} \label{eq_comp_pet_phat} \\
    |\hat{p}|_h &\leq \left(1 + \frac{\theta}{2} \right) |p \et|_h + \frac{29^2}{2\theta} \frac{\Gamma + x}{h n} \label{eq_comp_phat_pet},
\end{align}
where $\Gamma$ and $\Omega_x$ are as in Definition \ref{def_Omega}.
\end{prop}

Based on Theorem \ref{thm_or_ineq_mod} and the above proposition with $\theta = \frac{1}{2}$,
let us define the following universal penalty. Let $\log_-$ denote the negative part of the $\log$ function, and let $a > 0$ be some parameter. Let then

\begin{equation} \label{eq_def_pen}
    \pen_a(h) = 29 \sqrt{\frac{4}{3}} \sqrt{\frac{|\hat{p}|_h(\Gamma + a \log_{-}(h))}{ h n}} + \sqrt{\frac{4}{3}} 29^2 \frac{\Gamma + a \log_-(h)}{h n},
\end{equation}
where $\Gamma$ is defined by equation \eqref{eq_def_gam}.

Assume that for any model $m \in \cM$, there is an associated parameter $h_m > 0$, chosen such that $\hat{p}_{m}^{(h_m,C_{h_m})}$ satisfies an oracle inequality on model $m$ with fixed constant independent of $m$, i.e such that 
\[ \kappa_m(h_m) \geq \kappa_0 > 0 \] 
for some constant $\kappa_0$. In order to perform model selection, we need to control the behaviour of the tests $T^{(h)}(\gX,p,q)$ when $p,q$ belong to two different models. It may be that comparing two models is much harder (i.e, requires a much smaller value of $h$) than optimizing performance within a single model. For example, while it is feasible to optimize among piecewise constant functions on a given partition $\cI$, selecting among partitions $\cI$ is impossible in general since the set $(\1_{[a,b]})$ of indicator functions of intervals is non-separable in $L^\infty$.
To avoid such cases, we make the following assumption.

\begin{ass} \label{ass_kap_et}
There exists a constant
\[  \kappa_* = \inf_{m,m' \in \cM} \left\{\kappa_{m \cup m'}(h_m \wedge h_{m'}) \right\} > 0.  \]
\end{ass}

Qualitatively speaking, Assumption \ref{ass_kap_et} states that comparing $p,q$ belonging to $m,m'$ is not significantly harder than comparing $p_1,q_1$ belonging to the same model ($m$ or $m'$).
For example, this is always the case when models are nested.

\begin{rem}
If $\cM$ is totally ordered with respect to inclusion, then 
\[ \kappa_* = \inf_{m \in \cM} \{ \kappa_m(h_m)  \} \geq \kappa_0 > 0. \]
\end{rem}

\begin{proof}
Let $m,m' \in \cM$ and assume without loss of generality that $m' \subset m$. Since $\kappa_m$ is a non-increasing function, 
\[ \kappa_{m \cup m'} (h_m \wedge h_{m'}) = \kappa_m(h_m \wedge h_{m'}) \geq \kappa_m(h_m).  \]
This proves that $\kappa_* \geq \inf_{m \in \cM} \{\kappa_m(h_m)\}$.
On the other hand, taking $m = m'$ in Assumption \ref{ass_kap_et} yields $\kappa_* \leq \kappa_m(h_m)$.
\end{proof}

We construct a model selection procedure as follows.
\begin{df} \label{def_moshun}
Let
$C_\square: \mathbb{R}_+ \times \sP^2 \to \mathcal{C}$ be a collection of function parametrized by $h > 0$. For any $h > 0,$ 
let $T^{(h)} = T^{(h,C_h)}$ be the corresponding $T-$test statistic. Let $h_{\cM} = (h_m)_{m \in \cM}$ be a sequence of positive reals.
For any $p \in \gM$, let
\begin{equation} \label{eq_def_hp}
   h_p = \sup \left\{ h_m : m \in \cM, p \in m \right\}. 
\end{equation}
For any $p \in \gM$, let then
\begin{equation} \label{eq_def_TcM}
   T_{\cM}^{(h_{\cM},C_{\square})}(\gX,p) = \sup_{q \in \gM} \left\{ T^{(h_p \wedge h_q)}(\gX,p,q) - \pen_a(h_q) \right\} + \pen_a(h_p). 
\end{equation}
A model selection hun-estimator (moshun-estimator) is defined to be any random element 
\[ \hat{p}_{\cM}^{(h_{\cM},C_{\square})} \in \gM \] 
such that
\begin{equation} \label{eq_def_moshun}
   T_{\cM}^{(h_{\cM},C_{\square})}\left(\gX,\hat{p}_{\cM}^{(h_{\cM},C_{\square})} \right) \leq \inf_{p \in \gM} \left\{ T_{\cM}^{(h_{\cM},C_{\square})}(\gX,p) \right\} + \delta.   
\end{equation} 
\end{df}

As for the hun-estimator, we shall shorten the notation to $T_\cM, \hat{p}_\cM$ whenever the sequence $h_\cM$ and functions $C_\square$ are fixed. Given the assumptions discussed above, the moshun-estimator satisfies the following oracle inequality.

\begin{thm} \label{thm_mod_select_gen}
Let Assumption \ref{ass_kap_et} hold for a sequence $h_{\cM}$ and some $\kappa_* > 0$. Let $C_\square$ be such that 
\begin{equation} \label{eq_eps_moshun}
\sup_{h > 0} \varepsilon_{\gM}(h,C_h) \leq \varepsilon < 1. 
\end{equation}
For all $y \geq e$, on an event ($\Omega_{a \log y}$) with probability greater than $1 - \frac{2}{y^a}$,
\begin{equation}
    \begin{split}
        (1 - \eps) \kappa_* \normfty{\hat{p}_{\cM}^{(h_{\cM},C_{\square})} - p \et} &\leq \inf_{m \in \cM} \left\{ (2 + (1-\eps)\kappa_*) \inf_{p \in m} \{ \normfty{p - p \et} \} + 4 \pen_a(h_m) \right\} \\
        &\quad + 29 y \sqrt{\frac{2 a}{3 e n}} + 29^2 \frac{4}{\sqrt{3}} \frac{a y}{e n} + \delta.
    \end{split}
\end{equation}
\end{thm}


An interesting aspect of Theorem \ref{thm_mod_select_gen} is that the penalty only depends on $(h_m)_{m \in \cM}$ and on the fixed parameter $a$, but not on the number of models. In particular, the theorem also applies to countably infinite collections $\cM$. 

\subsection{Piecewise polynomials on regular dyadic partitions} \label{sec:dyad_mod_coll}
The key question concerning applications of Theorem \ref{thm_mod_select_gen} is for which collections of models Assumption \ref{ass_kap_et} holds. We have already seen that Assumption \ref{ass_kap_et} holds for nested model collections, however the assumption that models are nested is restrictive: it excludes classes of irregular partitions that one would like to use in order to adapt to potentially inhomogeneous or anisotropic smoothness of the target density. 

Perhaps unexpectedly, Assumption \ref{ass_kap_et} turns out to be significantly weaker than nestedness. If $h_m$ is chosen according to Lemma \ref{prop_lb_kap_m_poly} (for a value $h < \frac{h_0}{r^2}$), then Assumption \ref{ass_kap_et} holds over the class $m(r,\cI)$, where $r \in \mathbb{N}$ and $\cI$ belongs to the set of \emph{regular dyadic partitions}, i.e, partitions
\[ \cI(\gj) = \left\{ \prod_{i = 1}^d  [k_i 2^{-j_i}, (k_i + 1) 2^{-j_i}) : \gk = (k_1,\ldots,k_d) \in \Z^d \right\},  \]
for some $\gj \in \Z^d$. For completeness, we prove in Section A3.3 of the Supplementary Material \cite{Proofs} 
that the $\cI(\gj)$ are indeed partitions of $\R^d$, with the property that $\cI(\gj')$ refines $\cI(\gj)$ whenever $\gj' \geq \gj$.


Denote then
\[ \mathfrak{I}_d = \left\{ \cI(\gj) : \gj \in \Z^d \right\} \]
and
\begin{equation} \label{eq_def_dyad_mod_coll}
   \cM_{\gr} = \{ m_{dir}(\gr,\cI) : \cI \in \mathfrak{I}_d \}. 
\end{equation}
For any $m = m_{dir}(\gr,\cI(\gj)) \in \cM_{\gr}$, let
\begin{equation} \label{eq_def_hm}
   h_m =  \frac{\min_{I \in \cI(\gj)} \{\mu(I)\}}{(2\norm{\gr}_1^2)^{d} 4^{d+1}} = \frac{2^{-(j_1 + \ldots + j_d)}}{(2\norm{\gr}_1^2)^{d} 4^{d+1}}. 
\end{equation}
Up to a constant depending only on $\gr,d$, this choice of $h_m$ is the same as in Corollary \ref{cor_mod_poly} and yields the same (almost optimal) risk bound.

Let now $\cC = \cC_{rec}$ be the class of $d-$dimensional open rectangles with sides parallel to the axes, i.e
\[ \cC_{rec} = \{ \prod_{i = 1}^d (a_i,b_i) : a_i, b_i \in \R, a_i < b_i \}. \]
This class generates the Borel sigma-algebra, hence for any $h > 0$, $|\cdot|_h$ is a norm on $\sP$. The following theorem shows that the model collection $\cM_{\gr}$ satisfies Assumption \ref{ass_kap_et} for a constant $\kappa_*$ depending only on $\gr$ and $d$.

\begin{thm} \label{thm_ass_lb_kap_poly}
For all $m,m' \in \cM_{\gr}$ and $h_m,h_{m'}$ defined by equation \eqref{eq_def_hm},
\[ \kappa_{m \cup m'}(h_m \wedge h_{m'}) \geq \left[4\left(1 + 4\sqrt{\prod_{i = 1}^d (r_i+1)} \right) \right]^{-1}. \]
\end{thm}

In light of Theorem \ref{thm_mod_select_gen}, Theorem \ref{thm_ass_lb_kap_poly} implies that it is possible to perform \emph{model selection} on the model collection $\cM_{\gr}$, in the sense that, for a suitable $C_\square$, the model-selection estimator $\hat{p}_{\cM_\gr}^{(h_{\cM_\gr},C_\square)}$ given by Definition \ref{def_moshun}
performs as well as the best estimator in the collection $\{\hat{p}_m^{(h_m)} : m \in \cM_{\gr} \}$, up to a constant depending only on $d,\gr$.


\section{Rates under anisotropic smoothness} \label{sec-rates}

The oracle inequality satisfied by the hun-estimator (Theorem \ref{thm_or_ineq_mod}), together with the lower bound on $\kappa_m$ for polynomial models (Proposition \ref{prop_lb_kap_m_poly}) allow to recover minimax optimal rates on anisotropic Lipschitz spaces. Moreover, the model selection results in the previous section (Theorems \ref{thm_mod_select_gen} and \ref{thm_ass_lb_kap_poly}) imply that this can be done in an adaptive manner, as we now show.  

Let $\gbeta \in \R^d$ be a multi-index, and let $C^{\gbeta}$ denote the space of functions $f$ which admit partial derivatives $\frac{\partial^{k_i} f}{\partial x_i^{k_i}}$ at all orders $k_i \leq \lfloor \beta_i \rfloor$, and are such that the semi-norms
\[ |f|_{i, \beta_i} := \sup_{x \in \mathbb{R}^d} \sup_{t \in \mathbb{R}} 
\frac{1}{t^{\beta_i - \lfloor \beta_i \rfloor}} \left| \frac{\partial^{\lfloor \beta_i \rfloor} f}{\partial x_i^{\lfloor \beta_i \rfloor}}(x + t e_i) - \frac{\partial^{\lfloor \beta_i \rfloor} f}{\partial x_i^{\lfloor \beta_i \rfloor}}(x) \right| \]
are finite for all $i \in \{1,\ldots,d\}$, where $e_i$ denotes the standard basis of $\R^d$. Note that this only requires regularity along the coordinate directions, and in particular the cross-derivatives may fail to exist.

It is known that the minimax-optimal convergence rate on the class $C^{\gbeta}$ is 
\[\left(\frac{\log n}{n} \right)^{\frac{\beta}{2\beta + d}}, \beta \text{ s.t. } \frac{d}{\beta} = \sum_{i = 1}^d \frac{1}{\beta_i}. \] 
This follows from results of \citet{Lepski2013}. More precisely, $C^{\gbeta}$ is a special case of the function spaces introduced by \citet{Lepski2013}, with no independence structure and all smoothness measured in sup-norm. The results of this section could be extended to general anisotropic Nikolskii classes using the same Sobolev embedding theorem as in \citep{Lepski2013} (equation (A.8), resulting from \citet[Theorem 6.9]{Nikolskii1977}). As discussed in Section \ref{sec_proj_arg}, there is also reason to think that the hun-estimator can handle the independence structure assumption. The precise norms used to measure smoothness in this article are different from those in \citep{Lepski2013} and give a more precise result, as we shall see later.

Let $\cC = \cC_{rec}$ be the class of products of $d$ open intervals.
To prove that a moshun estimator attains the optimal rate when $p\et \in C^{\gbeta}$, an approximation result is needed in order to bound the term $\inf_{p \in m_{dir}(\gr,\cI(\gj))} \normfty{p - p\et}$. This is the purpose of the following proposition.  

\begin{prop} \label{prop_approx_poly}
Let $f \in C^{\gbeta}(\R^d) \cap L^1(\mathbb{R}^d)$.
Let $\gr \geq \lfloor \gbeta \rfloor$ be a vector of integers, let $\gh = (h_j)_{1 \leq j \leq d}$ be non-negative real numbers and let $\cI_{\gh}$ denote the rectangular partition
\[ \left\{ \prod_{j = 1}^d [k_j h_j, (k_j+1)h_j) : k \in \mathbb{Z}^d \right\}. \]
There exists $f_d \in \overline{m_{dir}}(\gr,\cI_{\gh})$ such that  
\[ \normfty{f - f_d} \leq 2 b_d(\gr) \max_{1 \leq j \leq d} \left\{ \frac{h_j^{\beta_j}}{\lfloor \beta_j \rfloor!} |f|_{j,\beta_j} \right\}, \]
where
\begin{equation} \label{eq_def_bd}
   b_d(\gr) = 1 + \min_{\sigma \in \mathfrak{S}_d} \left\{ \sum_{j = 1}^d \prod_{i = 1}^j \left[\frac{2}{\pi} \log(1 + r_{\sigma(i)}) + 1 \right] \right\}. 
\end{equation}
\end{prop}

Similar results have long been established in the approximation theory literature when $\gbeta \in \N^d$, along with bounds on the approximation error expressed in terms of finite difference operators - the article \citep{Dahmen1980} is particularly relevant. However, these results usually involve a non-explicit constant. Rather than adapt them to our setting, it is just as convenient to give a direct proof, which also provides an explicit constant.
Optimizing the bias-variance tradeoff between approximation error (given by Proposition \ref{prop_approx_poly}) and estimation error (given by $\pen_a(h_m)$ defined in equation \eqref{eq_def_pen}) yields the following theorem.

\begin{thm} \label{thm_rates}
Let $\gbeta \in \R_+^d$ be such that $\lfloor \beta_j \rfloor \leq r_j$ for all $j$. Let 
\[ \beta = \frac{1}{d} \sum_{j = 1}^d \frac{1}{\beta_j}. \]
Let $p\et = \frac{1}{n}\sum_{i = 1}^n p_i\et$. Assuming that $p\et \in C^{\beta}\left( \mathbb{R}^d \right)$, let
\[ L_{\gbeta}(p\et) = \prod_{j = 1}^d |p \et|_{j,\beta_j}^{\frac{\beta}{d \beta_j}}. \]
Let $h_{\cM_{\gr}} = (h_m)_{m \in \cM_{\gr}}$ be given by equation \eqref{eq_def_hm}. There exists some function $C_{\square}: \mathbb{R}_+ \times \sP^2 \to \cC_{rec}$ such that for any $n \in \mathbb{N}$ such that $L_{\gbeta}(p\et)^{\frac{d}{\beta}} \leq \frac{n}{\log n}$,
\[ \mathbb{E} \left[ \normfty{\hat{p}_{\cM_{\gr}}^{(h_{\cM_{\gr}},C_\square)} - p\et} \right] \leq  c \normfty{p\et}^{\frac{\beta}{2\beta + d}} L_{\gbeta}(p\et)^{\frac{d}{2\beta + d}} \left( \frac{\log n}{n} \right)^{\frac{\beta}{2\beta + d}} \]
when $\normfty{p \et} \geq L_{\gbeta}(p\et)^{\frac{d}{\beta + d}} \left( \frac{\log n}{n} \right)^{\frac{\beta}{\beta + d}}$, and
\[ \mathbb{E} \left[ \normfty{\hat{p}_{\cM_{\gr}}^{(h_{\cM_{\gr}},C_\square)} - p\et} \right] \leq c L_{\gbeta}(p\et)^{\frac{d}{\beta + d}} \left( \frac{\log n}{n} \right)^{\frac{\beta}{\beta + d}}  \]
otherwise, where $c$ is a constant which depends only on $\gr,d$.
\end{thm}

Assume to simplify the discussion that $\gp \et = (p\et,\ldots,p\et))$ and that $p\et \in C^{\gbeta}(\R^d)$. Then, Theorem \ref{thm_rates} yields the correct \citep[Theorem 2]{Lepski2013} minimax convergence rate, $\bigl( n^{-1} \log n \bigr)^{\beta/(2\beta + d)}$, with respect to the sample size $n$. Moreover, the dependence of the upper bound on $p\et$ is explicitly expressed in terms of $\normfty{p\et}$ and the semi-norms $|p\et|_{j,\beta_j}$. The assumption $p\et \in C^{\gbeta}(\R^d)$, together with the fact that $p\et$ is a density, imply a bound on $\normfty{p \et}$ by a function of $\gbeta,d$ and the semi-norms $|p\et|_{j,\beta_j}$ (as detailed below). However, $\normfty{p\et}$ can be much smaller than this upper bound, even for fixed $|p\et|_{j,\beta_j} = L_j$, for example (in dimension $1$) if $p\et$ oscillates between $0$ and a small $t > 0$ on an interval of size $\approx \frac{1}{t}$. Theorem \ref{thm_rates} yields an improved bound in that case. In particular, $\hat{p}$  converges faster in the neighbourhood of $0$ (which is not a density, but is a uniform limit of densities). 

Consider now the class of functions
\[ \cC_{\gL,b}^{\gbeta} = \left\{ p \in C^{\gbeta}(\R^d) : \normfty{p} \leq b, \forall j \in \{1,\ldots,d \}, |p|_{j,\beta_j} \leq L_j \right\}, \]
as well as the class $\cP_{\gL,b}^{\gbeta}$ of probability density functions which belong to $\cC_{\gL,b}^{\gbeta}$.
Compared to the constants $Q_j$ which appear in \citet{Lepski2013}, the $L_j$ only constrain the Hölder semi-norms of maximal order $\beta_j$. An independent constant ($b$) bounds the sup-norm of the density $p\et$. As discussed above, these two quantities can indeed vary independently to some extent. Qualitatively, they play different roles, as $b$ only intervenes in the estimation error while the constants $L_j$ only affect the approximation error.

The function classes $\cP_{\gL,b}^{\gbeta}$ are non-decreasing as a function of $b$, moreover by Lemma A4.2 from the Supplementary Material \citep{Proofs}, 
there is a function $b_{max}(\gL,\gbeta)$ such that $b \mapsto \cP_{\gL,b}^{\gbeta}$ is strictly increasing for $b < b_{max}(\gL,\gbeta)$ and constant for all $b \geq b_{max}(\gL,\gbeta)$. Theorem \ref{thm_rates} implies the following minimax upper bound on the classes $\cP_{\gL,b}^{\gbeta}$:

\begin{cor} \label{cor_ubd_minmax_risk_aniso}
Let $\gbeta \in \R_+^d$ and $\frac{1}{\beta} = \frac{1}{d} \sum_{j = 1}^d \frac{1}{\beta_j}$. Let $\gL \in \mathbb{R}_+^d$ and
\[ L = \prod_{j = 1}^d L_j^{\frac{\beta}{d \beta_j}}. \]
For all $b \leq b_{max}(\gL,\gbeta)$ and all large enough $n$,
\[ \inf_{\tilde{p}} \sup_{p\et \in \cP_{\gL,b}^{\gbeta}} \E \left[ \normfty{\hat{p} - p\et} \right] \leq C b^{\frac{\beta}{2\beta + d}} L^{\frac{d}{2\beta + d}} \left( \frac{\log n}{n} \right)^{\frac{\beta}{2\beta + d}}, \]
where the infimum runs over all estimators computed from an $n-$sample drawn from $p\et$, and $C$ is a constant depending only on $\gbeta,d$.
\end{cor}

The above rate of convergence is known to be optimal. As we shall now see, the above minimax upper bound is also optimal with respect to the constants $b,\gL$. A matching lower bound may be derived from a careful reading of the proof of \citet[Theorem 2]{Lepski2013},
resulting in the following theorem. 

\begin{thm} \label{thm_minimax_lbd_ani_smooth}
Let $\gbeta \in \R_+^d$ and $\frac{1}{\beta} = \frac{1}{d} \sum_{j = 1}^d \frac{1}{\beta_j}$. Let $p_b$ denote the isotropic, centered Gaussian pdf with norm $\normfty{p_b} = b$.
For all $L \in \R_+^d$ and all $b > 0$ such that $p_b \in C_{\frac{\gL}{2},+\infty}^{\gbeta}$ and all large enough $n$,
\[ \inf_{\tilde{p}} \sup_{p\et \in \cP_{\gL,b}^{\gbeta}} \E \left[ \normfty{\tilde{p} - p\et} \right] \geq C b^{\frac{\beta}{2\beta + d}} \left(\prod_{j = 1}^d L_j^{\frac{\beta}{\beta_j}} \right)^{\frac{1}{2\beta + d}} \left( \frac{\log n}{n} \right)^{\frac{\beta}{2\beta + d}}, \]
where the infimum runs over all estimators computed from an $n-$sample drawn from $p\et$, and $C$ is a constant depending only on $\gbeta,d$.
\end{thm}

Thus, the moshun-estimator adapts not only to the smoothness $\gbeta$ but also to the constants $\left(L_j \right)_{1 \leq j \leq d}$ and $b$, a more precise result than obtained in \citet{Lepski2013} (where the constants $Q_j$ bound both $L_j$ and $b$ along with intermediate derivatives). This property has not, to the best of our knowledge, been established for any estimator for density estimation in sup-norm, though such a result was known in the setting of white noise regression on $[0,1)^d$ under a Hölder regularity assumption \citep{bertin2004}.  

\section{Acknowledgements}
The author would like to thank the anonymous referees for their constructive suggestions which helped to improve the paper. The author would also like to acknowledge the support of ENSAI, where the last steps of the revision were carried out. 

\section{Funding}
The project received funding from the European Union's Horizon 2020 research program under grant agreement N° 811017.


\bibliographystyle{abbrvnat} 
\bibliography{norme_infinie.bib}

\end{document}